\documentclass[final]{dmtcs-episciences}

\usepackage[utf8]{inputenc}
\usepackage{subfigure}

\usepackage{hyperref}
\usepackage{amsmath}
\usepackage{url}
\usepackage{tikz}

\newtheorem{proposition}{Proposition}
\newtheorem{conjecture}{Conjecture}
\newtheorem{theorem}{Theorem}
\newtheorem{observation}{Observation}
\newtheorem{corollary}{Corollary}

\newcommand{\asc}{{\rm asc}\,}

\newcommand{\tp}{{\rm top}\,}
\newcommand{\bottom}{{\rm bottom}\,}

\newcommand{\e}{\mathbf{e}}

\newcommand{\bb}{\mathbf{b}}

\newcommand{\Sn}{\mathbf{S}}
\newcommand{\I}{\mathbf{I}}
\newcommand{\TT}{\mathcal{T}}

\tikzset{dot/.style={draw,shape=circle,fill=black,scale=.7}}

\author{Sylvie Corteel\affiliationmark{1}\thanks{Supported in part by grant ANR-08-JCJC-0011}
  \and Megan A. Martinez\affiliationmark{2}
  \and Carla D. Savage\affiliationmark{3} \\
  \and Michael Weselcouch\affiliationmark{4}}
\title{Patterns in Inversion Sequences I}
\affiliation{
  LIAFA, CNRS, Universit\'e Paris Diderot 7, Paris, France\\
  Department of Mathematics, Ithaca College, Ithaca, NY\\
  Department of Computer Science, North Carolina State University, Raleigh, NC\\
  Department of Mathematics, North Carolina State University, Raleigh, NC}
\received{2015-11-27}
\accepted{2016-1-7}
\begin{document}
\publicationdetails{18}{2016}{2}{2}{1325}
\maketitle
\begin{abstract}
  Permutations that avoid given patterns have been studied in great depth for their connections to other fields of mathematics, computer science, and biology.  From a combinatorial perspective, permutation patterns have served as a unifying interpretation that relates a vast array of combinatorial structures.  
In this paper, we introduce the notion of patterns in inversion sequences.  A sequence $(e_1,e_2,\ldots,e_n)$ is an inversion sequence if $0 \leq e_i<i$ for all $i \in [n]$.  Inversion sequences of length $n$ are in bijection with permutations of length $n$; an inversion sequence can be obtained from any  permutation $\pi=\pi_1\pi_2\ldots \pi_n$ by setting $e_i = |\{j \ | \ j < i \ {\rm and} \ \pi_j > \pi_i \}|$.  This correspondence makes it a natural extension to study patterns in inversion sequences much in the same way that patterns have been studied in permutations.  This paper, the first of two on patterns in inversion sequences, focuses on the enumeration of inversion sequences that avoid words of length three.  Our results connect patterns in inversion sequences to a number of well-known numerical sequences including Fibonacci numbers, Bell numbers, Schr\"oder numbers, and Euler up/down numbers. 
\end{abstract}



\section{Overview}

A permutation $\pi=\pi_1\pi_2\ldots \pi_n \in \Sn_n$ is said to \emph{contain} a pattern $\sigma=\sigma_1\sigma_2\ldots\sigma_k \in \Sn_k$ if there exist indices $i_1<i_2<\cdots<i_k$ such that for every $a,b \in \{1,2,\ldots,k\}$, we have $\pi_{i_a}<\pi_{i_b}$ if and only if $\sigma_a<\sigma_b$.  Otherwise, $\pi$ is said to \emph{avoid} the pattern $\sigma$.  Let $\Sn_n(\sigma)=\{ \pi \in \Sn_n \mid \pi \text{ avoids } \sigma \}$.  Given any $\sigma \in \Sn_k$, the \emph{avoidance sequence} of $\sigma$ is the integer sequence, \[|\Sn_1(\sigma)|, |\Sn_2(\sigma)|, |\Sn_3(\sigma)|, \ldots.\] The avoidance sequences for various $\sigma$ count a great number of well-known combinatorial structures.  As a result, the study of permutation patterns provides a unifying interpretation for a number of disparate discrete structures.  Early work of MacMahon  \cite{macmahon} enumerating permutations avoiding 123 and of Knuth \cite{knuth1,knuth3} on 231 and stack-sortable permutations  established by 1973 that for every $\sigma \in\Sn_3$, 
$\Sn_n(\sigma)$ is counted by the Catalan numbers.
In  1985 Simion and Schmidt published the first systematic study of pattern avoidance in permutations \cite{SimionSchmidt}.
Thirty years later, there is a rich body of work demonstrating the connections between pattern avoidance and many areas of mathematics and computation.  See, for example, the survey of Kitaev \cite{kitaev}.

Recently, the $s$-inversion sequences, $\I_n^{(s)}$, defined for a positive integer sequence $s=(s_1, \ldots, s_n)$ by
$$\I_n^{(s)} = \{ (e_1,e_2,\ldots, e_n )\ | \ 0 \leq e_i < s_i\},$$ were introduced in \cite{SS} to enumerate certain families of partitions  via Ehrhart theory.   For a variety of sequences $s$, natural statistics on  $\I_n^{(s)}$ have the same distribution as natural statistics on 
other equinumerous combinatorial families,  a phenomenon that was used to settle an open question about Coxeter groups in \cite{SV}.

The words of length $n$ over the alphabet $\{0,1, \ldots,k-1\}$ can be viewed as the inversion sequences $\I_n^{(k,k, \ldots,k)}$.
When $s=(1,2,\ldots,n)$, the $s$-inversion sequences $\I_n = \I_n^{(1,2, \ldots, n)}$ have been used in various ways to encode permutations $\Sn_n$.  
For example, the map  $\Theta(\pi):  \Sn_n \rightarrow \I_n$ defined for  $\pi=\pi_1 \ldots \pi_n \in \Sn_n$ by $\Theta(\pi)=(e_1,e_2, \ldots, e_n)$, where $e_i = |\{j \ | \ j < i \ {\rm and} \ e_j > e_i \}|$,
 is a bijection with several nice properties.
These connections to words and permutations make it  natural  to study pattern avoidance in inversion sequences in the same way that pattern avoidance has been studied in words and permutations.

Given a word $p=p_1p_2\ldots p_k \in \{0,1,\ldots,k-1\}^k$, define the \emph{reduction} of $p$ to be the word obtained by replacing the $i$th smallest entries of $p$ with $i-1$.  For instance, the reduction of $3052662$ is $2031441$.  We say that an inversion sequence $e \in \I_n$ \emph{contains} the pattern $p$, if there exist some indices $i_1<i_2<\cdots<i_k$ such that the reduction of $e_{i_1}e_{i_2}\ldots e_{i_k}$ is $p$.  Otherwise, $e$ is said to \emph{avoid} $p$.
 Let $\I_n(p)=\{ e\in \I_n \mid e \text{ avoids } p \}$.  The \emph{avoidance sequence} of $p$ is the integer sequence \[|\I_1(p)|, |\I_2(p)|, |\I_3(p)|, \ldots.\]

Our focus in this paper is a study of inversion sequences avoiding a three-letter word $p$. 
We discover a surprisingly rich collection of enumerative results, as well as intriguing conjectures.  
There will be several examples where pattern-avoiding inversion sequences provide more natural models of combinatorial sequences than previously known.


This paper is one of the 
first\footnote{When our paper was first posted to the arXiv we were notified by Mansour that he and Shattuck had independently obtained results on  $|\I_n(\sigma)|$ for the patterns
$\sigma=012,021,102,201,210$ in \cite{Mansour}.  Their methods are quite complementary to ours, as we will describe in Section 2.}
systematic studies of pattern avoidance in inversion sequences,
although in \cite{DS} Duncan and Steingr{\'i}msson considered pattern avoidance in {\em ascent sequences}, introduced in \cite{Bousquet},  and obtained interesting enumerative results.
(An ascent sequence is an integer sequence $(e_1,e_2,\ldots,e_n)$ in which $e_1=0$ and $e_i$ is a nonnegative integer at most one more than the number of ascents in the sequences $(e_1,e_2,\ldots,e_{i-1})$.
The ascent sequences of length $n$ thus form a subset of $\I_n$.)
Some of the open questions posed in \cite{DS} were settled by Mansour and Shattuck in \cite{MS}.



Our approach in this paper is combinatorial.  For each pattern $p$,  we first observe 
an alternate characterization of the $p$-avoiding inversion sequences.  
For example:  {\em the inversion sequences avoiding 011 are those whose positive elements are distinct} (see Section 3.3).
We then use the observation to define the structure of the $p$-avoiding inversion sequences and  relate them to equinumerous combinatorial families via bijections, recurrences or generating functions.
 In the process, we  discover and prove refinements via correspondences between natural statistics.  For example, the number of 011-avoiding inversion sequences in $\I_n$ with $k$ {\em zeros} is equal to the number of partitions of an $n$-element set into $k$ {\em nonempty blocks} (Section 3.3).

In Section 2 we consider inversion sequences  $\I_n$ avoiding a given 3-letter permutation of $\{0,1,2\}$. We show that the number of inversion sequences avoiding {\bf 012} is given by the {\bf odd-indexed Fibonacci numbers} and that the inversion sequences avoiding {\bf 021}  are counted by the {\bf large Schr\"oder numbers}.
We prove that the patterns {\bf 201} and {\bf 210} are {\em Wilf equivalent}, i.e. that they have the same avoidance sequence.
This  sequence does not appear in the OEIS, but we derive a recurrence for it.
Enumerative results for the patterns  012, 021, 201, 102, and 210 also appear in \cite{Mansour}, where 
Mansour and Shattuck  additionally prove that
the number of inversion sequences avoiding {\bf 102} is given by the sequence A200753 in the On-Line Encyclopedia of Integer Sequences (OEIS) \cite{Sloane}.
The number of inversion sequences avoiding {\bf 120} does not appear in the OEIS and counting them remains an open problem.

In Section 3 we consider inversion sequences  $\I_n$ avoiding a given 3-letter pattern with repeated symbols.
We prove that  the inversion sequences avoiding {\bf 000} are counted by the {\bf Euler up/down numbers}.  
We show that the inversion sequences avoiding {\bf 001} are counted by {\bf powers of two}, and those avoiding {\bf 011} are counted by the {\bf Bell numbers}.
Additionally, we prove that the number of inversion sequences avoiding {\bf 101} is the same as the number avoiding
{\bf 110} and this is the same as the number of {\bf permutations avoiding the vincular pattern 1-23-4}.
The avoidance sequence for the pattern {\bf 010} does not appear in the OEIS 
nor does the number of inversion sequences avoiding {\bf 100}.  Counting either of these sets is an open problem.

In Section 4, we return to {\bf 021}-avoiding inversion sequences.  We examine the correspondence between $\I_n(021)$ and Schr\"oder $(n-1)$-paths, and between  $\I_n(021)$ and certain binary trees by introducing two  further bijections.  Each bijection succeeds in relating a variety of different statistics in inversion sequences and those combinatorial families.  Surprisingly, the ascent statistic in $\I_n(021)$ is {\em symmetric}, and we prove this by defining a bijection between $\I_n(021)$ and a tree structure that is known to be counted by the Schr\"oder numbers.  Section 4 is a testament to the rich combinatorial structure that can be uncovered when examining a class of pattern-avoiding inversion sequences in-depth.

The results on inversion sequences avoiding permutations and words of length 3 are summarized in Table~\ref{permswords}.
In the last column of the table, we use $|\I_7(p)|$ as an identifier for the avoidance sequence of $\I_n(p)$.

This paper is the first part of a larger study on patterns in inversion sequences.
In subsequent work, we consider a generalization of pattern-avoiding inversion sequences that includes many more surprising results and conjectures.

Throughout this paper, we let $[n]=\{1,2,\ldots, n \}$.  For an inversion sequence $e=(e_1,e_2, \ldots, e_n) \in \I_n$,
 let $\sigma_t(e)=
 (e_1',e_2', \ldots, e_n')$ where $e_i'=0$ if $e_i=0$ and otherwise $e_i'=e_i+t$ (we allow for negative $t$).  This operation will also be applied to  substrings of an inversion sequence.  
 
 We use concatenation to add an element to the beginning or end of an inversion sequence: $0\cdot e$ is the inversion sequence $(0,e_1,e_2, \ldots, e_n)$ and for $0 \leq i \leq n$, $e \cdot i$ is the inversion sequence
 $(e_1, e_2, \ldots, e_n, i)$.

\begin{table}
\hrule 
\vspace{.1in}
\begin{tabbing}
XXXXXxxxxxx\=xxxxxxxxxxxxxxxxxxxxxxxxxxxxxx\=xxxxxxxxxxxxxxx\= xxxxxxx \= \kill

Pattern  $p$\> \ \ \  $a_n = |\I_n(p)|$ counted by: \>in OEIS? \>  sequence identifier $a_7$\\
\\

012 \> $F_{2n-1}$  \  \ (Boolean permutations)\> A001519 \> \>233\\
021 \>  Large Schr\"oder numbers \> A006318 \>\> 1806\\
102 \>  $A(x)=1+(x-x^2)(A(x))^3$   \cite{Mansour} \> A200753 \>\> 1694\\
120 \> ? \> no  (A263778)\> \> 2803\\
201 \>  Theorem \ref{rec:210ab} and \eqref{count210} \> no  (A263777)\> \> 4306\\
210 \>  Theorem \ref{rec:210ab} and \eqref{count210} \> no  (A263777)\> \> 4306\\
\\
000 \> Euler up/down numbers \> A000111 \> \> 1385\\
001 \> $2^{n-1}$  \ \ \ ($|\Sn_n(132,231)|$)\> A000079  \> \> 64\\
010 \> ? \> no (A263779) \> \> 979\\
100 \> ? \> no  (A263780)\> \> 3399\\
011 \> Bell numbers \> A000110 \> \> 877\\
101 \> $|\Sn_n$(1-23-4)$|$ \> A113227 \> \> 3207\\
110 \> $|\Sn_n$(1-23-4)$|$ \> A113227 \> \> 3207
\end{tabbing}
\hrule

\caption{Enumeration of inversion sequences avoiding permutations and words. (OEIS numbers in parentheses were newly assigned after this paper was posted to the arXiv)}
\label{permswords}
\end{table}

\section{Inversion sequences avoiding permutations}
\subsection{Avoiding {\bf 012}:  $F_{2n-1}$ and Boolean permutations}

Let $F_n$ denote the $n$th Fibonacci number, where $F_0=0$, $F_1=1$  and for $n \geq 2$, $F_n=F_{n-1}+F_{n-2}$.
Note that $a_n=F_{2n-1}$ satisfies the recurrence  $a_n=3a_{n-1}-a_{n-2}$, with initial conditions $a_1=1$, $a_2=2$.

Permutations avoiding both  $321$ and $3412$  are known as {\em Boolean permutations}
\cite{Ptenner,tenner}
 and are counted by  $F_{2n-1}$.  In this section we show that $\I_n(012)$ is the number of Boolean permutations of $[n]$.

\begin{observation}
The inversion sequences avoiding $012$ are those whose positive elements form a weakly decreasing sequence.
\label{observation:012}
\end{observation}

\begin{theorem}
$|\I_n(012)| = F_{2n-1}$.
\end{theorem}
\begin{proof}
 By Observation \ref{observation:012},  if $e \in \I_{n-1}(012)$, the following are all in $\I_n(012)$:  $e \cdot 0$, $e \cdot 1$, and $0 \cdot \sigma_1(e)$.
Every element in $\I_{n}(012)$ arises from an element of $\I_{n-1}(012)$ in at least one of these three ways, but certain elements are counted twice, namely those of the form $0 \cdot x \cdot 0$ where $x \in \I_{n-2}(012)$.
Thus $|\I_n(012)|$ satisfies the recurrence $a_n=3a_{n-1}-a_{n-2}$, with initial conditions $a_1=1$, $a_2=2$.  This is the same
recurrence satisfied by  $F_{2n-1}$.
\end{proof}

In view of the connection between Boolean permutations  and Coxeter groups highlighted in the recent paper of Petersen and Tenner \cite{Ptenner}, it would be nice to have a simple bijection encoding them as $012$-avoiding inversion sequences.

\subsection{Avoiding {\bf 021}:   the large Schr\"oder numbers}

A {\em Schr\"oder  $n$-path } is a path in the plane from $(0,0)$ to $(2n,0)$, never going below the $x$-axis, using only the steps
$(1,1)$ (up), $(1,-1)$ (down) and  $(2,0)$ (flat).  For example, using $U$, $D$, and $F$ for up, down, and flat steps, respectively,  the Schr\"oder 14-path
$p=UUDUFUDUFDDDUUDUDDUUUFDDD$ is shown in Figure \ref{Schroeder Path}.

\tikzset{sdot/.style={draw,shape=circle,fill=black,scale=.3}}
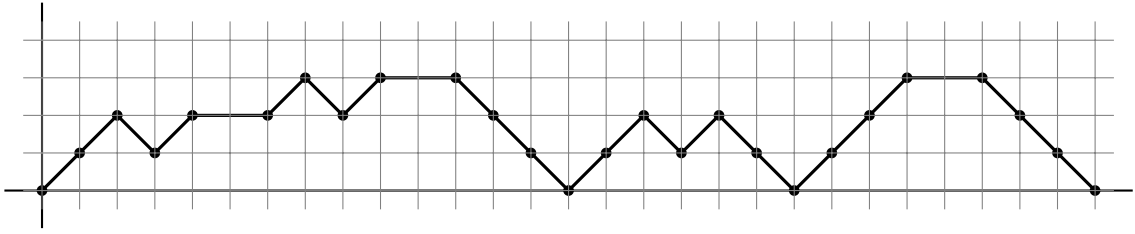
\begin{figure}
\begin{tikzpicture}[scale=.5]
\draw[very thick] (0,0) node[sdot]{}--(1,1) node[sdot]{}--(2,2) node[sdot]{}--(3,1) node[sdot]{}--(4,2) node[sdot]{}--(6,2) node[sdot]{}--(7,3) node[sdot]{}--(8,2) node[sdot]{}--(9,3) node[sdot]{}--(11,3) node[sdot]{}--(12,2) node[sdot]{}--(13,1) node[sdot]{}--(14,0) node[sdot]{}--(15,1) node[sdot]{}--(16,2) node[sdot]{}--(17,1) node[sdot]{}--(18,2) node[sdot]{}--(19,1) node[sdot]{}--(20,0) node[sdot]{}--(21,1) node[sdot]{}--(22,2) node[sdot]{}--(23,3) node[sdot]{}--(25,3) node[sdot]{}--(26,2) node[sdot]{}--(27,1) node[sdot]{}--(28,0) node[sdot]{};
\draw[thick] (-1,0)--(29,0);
\draw[thick] (0,-1)--(0,5);
\draw[thin, gray] (-.5,-.5) grid (28.5,4.5);
\end{tikzpicture}
\caption{The Schr\"oder 14-path $p=UUDUFUDUFDDDUUDUDDUUUFDDD$} \label{Schroeder Path}
\end{figure}

Let $R_n$ denote the set of Schr\"oder $n$-paths  and let $r_n = |R_n|$. It is well known that the generating function for $r_n$ is

\begin{equation}
R(x)  =\sum_{n =0}^{\infty} r_nx^n =  \frac{1-x-\sqrt{x^2-6x+1}}{2x}.
\label{eqn:Rgf}
\end{equation}
To relate Schr\"oder paths to inversion sequences, 
first note that the $021$-avoiding inversion sequences have the following simple characterization.
\begin{observation}
An inversion sequence avoids $021$ if and only if its positive entries are weakly increasing.
\label{observation:021}
\end{observation}

Using this characterization, we show that the elements of $\I_n(021)$ are in bijection with the Schr\"oder paths in $R_{n-1}$.  

\begin{theorem}

For $n \geq 1$, $\left | \I_n(021) \right | = r_{n-1}$.
\label{theorem:021}
\end{theorem}
\begin{proof}  
Let $E(x) = \sum_{n=1}^{\infty} |\I_n(021)| x^n$.
We show that $E(x)$ satisfies
\begin{equation}
E(x) = x + xE(x) + E^2(x),
\label{gf021}
\end{equation}
which has solution
\[
E(x) = \frac{1-x-\sqrt{x^2-6x+1}}{2} = xR(x),
\]
and the result will follow from \eqref{eqn:Rgf}.

Given $e \in \I_n(021)$, consider the last position $j+1$ such that $e_{j+1}$ attains its maximal value $j$.
If $j$=0, then either $e=(0)$ or $e=0 \cdot e'$ for some $e' \in \I_{n-1}(021)$.  This accounts for the  $x+xE(x)$ in
\eqref{gf021}.  We show that the $e$ for which $j>0$ are counted by $E^2(x)$.

If $j > 0$, then $(e_1, \ldots, e_j) \in \I_j(021)$.   As for $(e_{j+2}, \ldots, e_n)$, we know that for $i = 2, \ldots, n-j$, by Observation~\ref{observation:021} and by choice of $j$, either $e_{j+i}=0$ or
$$j \leq e_{j+i} < j+i-1.$$
It follows that subtracting $j-1$ from the positive entries of $(e_{j+2}, \ldots, e_n)$ and adding a $0$ to the beginning of the sequence gives an element of $\I_{n-j}(021)$.  Conversely, for any sequences $(e_1, \ldots, e_j) \in \I_j(021)$ and $f\in \I_{n-j}(021)$,  if we add $j-1$ to all the positive elements  of $f$ and remove the initial 0 of $f$, we can append the result to $(e_1, \ldots, e_{j},j)$ to obtain an element of
$\I_n(021)$  in which $ j+1$ is the last position whose entry attains its maximal value.
\end{proof}

We recall the operation $\sigma_k$ defined in Section 1 to 
 facilitate describing a bijection. If $e$ is an inversion sequence, and $(e_i,e_{i+1}, \ldots, e_j)$ is a substring of $e$ in which all positive entries are larger than $k$, then $\sigma_{-k}(e_i,e_{i+1}, \ldots, e_j)$ is the sequence obtained by subtracting $k$ from the positive entries of $(e_i,e_{i+1},\ldots,e_j)$.

The proof of Theorem \ref{theorem:021} gives rise to the following recursive bijection $\rho$ from $\I_n(021)$ to Schr\"oder $(n-1)$-paths.  For $e \in \I_n(021)$,  let $j+1$ be the last position such that $e_{j+1}$ attains its maximal value $j$.  Set $\rho(0)$ to be the empty path.  Then the Schr\"oder path $p = \rho(e)$ is defined by
\[
\rho(e) = \left \{
\begin{array}{ll}
F \rho(e_2, \ldots, e_n) & {\rm if}  \ j=0\\
U \rho(e_1, \ldots, e_{j}) D \rho(\sigma_{1-j}(0 \cdot (e_{j+2}, \ldots, e_n))) & {\rm otherwise}.
\end{array}
\right .
\]
For example, $p= \rho(0,1,0,1,0,2,5,7,7,7,9,0,10,11,12)$ is the Schr\"oder 14-path
 shown in Figure \ref{Schroeder Path}.  By definition, $\rho$ gives the following result.
\begin{theorem}
For $n \geq 1$,
the number of  $021$-avoiding inversion sequences of length $n$ with $k$  maximal elements is the same as the number of  Schr\"oder $(n-1)$-paths   with $k-1$ initial up steps.  
\end{theorem}

The number of Schr\"oder $(n-1)$-paths with $k$ initial up steps is counted by sequence A132372 in the OEIS \cite{Sloane}.

Ira Gessel considered in  \cite{gessel}   the generating function
\[
R(x,z) = \sum_{n=0}^{\infty} \sum_{p \in R_n} x^nz^{{\rm flat}(p)},
\]
where ${\rm flat}(p)$ is the number of flat steps in a Schr\"oder path $p$.  He showed that
\begin{equation}
R(x,z) = \frac{1-xz-\sqrt{(1-xz)^2-4x}}{2x}.
\label{eqn:Rzgf}
\end{equation}
Similarly define
\[
E(x,z) = \sum_{n=1}^{\infty} \sum_{e \in \I_n(021)} x^nz^{{\rm zeros}(e)},
\]
where ${\rm zeros}(e)$ is the number of zeros in the inversion sequence $e$.
Following the  proof  of Theorem  \ref{theorem:021},  we have
\[
E(x,z) = xz+xzE(x,z) + E^2(x,z)/z.
\]
Solving, and comparing to \eqref{eqn:Rzgf}  we have
\[
E(x,z) = \frac{z(1-xz-\sqrt{(1-xz)^2-4x})}{2} = xzR(x,z).
\]
This proves the following.
\begin{theorem}
For $n \geq 1$, the number of  $021$-avoiding inversion sequences of length $n$ with $k$ zeros is the same as the number of  Schr\"oder $(n-1)$-paths with $k-1$ flat steps.
\end{theorem}

Note that the number of Schr\"oder $(n-1)$-paths with $k-1$ flat steps is equal to the number of Schr\"oder $(n-1)$-paths with $k-1$ peaks (where a peak is an occurrence of $UD$).  This can be seen by applying an involution that replaces every $UD$ with $F$ and every $F$ with $UD$.  This gives us the following.

\begin{corollary}
For $n \geq 1$, the number of  $021$-avoiding inversion sequences of length $n$ with $k$ zeros is the same as the number of  Schr\"oder $(n-1)$-paths with $k-1$ peaks.
\end{corollary}

In Section 4 we provide a second bijection  relating different statistics in 021-avoiding inversion sequences and Schr\"oder paths.
We also show a correspondence with certain trees counted by the Schr\"oder numbers and use this to prove that the ascent statistic on 021-avoiding inversion sequences is symmetric.

\subsection{The patterns {\bf 201} and {\bf 210}}
 
 In this section we prove that the patterns $201$ and $210$ are Wilf equivalent on inversion sequences.
 This has also been shown in \cite{Mansour}.  The avoidance sequence for these patterns did not appear in the OEIS, (it has now been assigned A263777)  but we derive a recurrence to compute it.

For $e \in \I_n$, call position $j$ a {\em weak left-to-right maximum} if $e_i \leq e_j$ for all
$1 \leq i < j$.

\begin{observation} The $210$-avoiding inversion sequences are precisely those
that can be partitioned into two weakly increasing subsequences. 
\label{observation:210}
\end{observation}
\begin{proof}
Suppose $e$ has such a partition $e_{a_1} \leq e_{a_2} \leq \cdots \leq e_{a_t}$ and
$e_{b_1} \leq e_{b_2} \leq \cdots \leq e_{b_{n-t}}$. 
 If there exists $i<j<k$ such that $e_i > e_j > e_k$, then no two of $i,j,k$ can both be in  $\{a_1, \ldots, a_t\}$ or both be in  $\{b_1, \ldots, b_{n-t}\}$, so $e$ avoids $210$. Conversely, if $e$ avoids $210$, let $a=(a_1, \ldots, a_t)$ be the sequence of weak left-to-right maxima of $e$.
Then $e_{a_1} \leq e_{a_2} \leq \cdots \leq e_{a_t}$.  Consider $i,j \not\in \{a_1, \ldots, a_t\}$ where $i<j$.  The fact that $e_i$ is not a weak left-to-right maxima implies there exists some $e_s$ such that $s<i$ and $e_s>e_i$.  Thus to avoid 210, we must have $e_i \leq e_j$.
\end{proof}

\begin{observation}
Let $(e_1,e_2,\ldots,e_n) \in \I_n$.  Additionally, for any $i \in [n]$, let $M_i=\max(e_1,e_2,\ldots,e_{i-1})$.  Then $e \in \I_n(201)$ if and only if for every $i \in [n]$, the entry $e_i$ is a weak left-to-right maximum, or for every $j>i$, we have $e_j \leq e_i$ or $e_j>M_i$.

\label{observation:201}
\end{observation}
\begin{proof}
Let $e \in \I_n$ satisfy the conditions of Observation \ref{observation:201} and, to obtain a contradiction, assume there exist $i<j<k$ such that $e_ie_je_k$ forms a 201 pattern (i.e. $e_j<e_k<e_i$).  Notice that $M_j=\max\{e_1,e_2,\ldots,e_{j-1}\} \geq e_i$.  It follows that $M_j>e_k>e_j$, which contradictions our assumption.

Conversely, if $(e_1,e_2,\ldots,e_n) \in \I_n(201)$, consider any $e_i$.  If $e_i$ is not a weak left-to-right maximum, then there exists some maximum value $M_i=e_s$ such that $s<i$ and $e_s>e_i$.  Therefore, in order to avoid a 201 pattern, any $e_j$ where $j>i$ must have $e_i \geq e_j$ or $e_j \geq M_i=e_s$.
\end{proof}

\begin{theorem}
For $n \geq 1$, $|\I_n(210)|=|\I_n(201)|$.
\label{210eq201}
\end{theorem}
\begin{proof}
We exhibit a bijection based on the characterizations in Observations \ref{observation:210} and \ref{observation:201}.

Given $e \in \I_n(210)$, define $f \in \I_n(201)$ as follows.
Let $e_{a_1} \leq e_{a_2} \leq \cdots \leq e_{a_t}$ be the sequence of  weak left-to-right maxima of $e$ and let
$e_{b_1} \leq e_{b_2} \leq \cdots \leq e_{b_{n-t}}$ be the subsequence of remaining elements of $e$.

For $i= 1, \ldots, t$, set $f_{a_i}=e_{a_i}$.  For each $j=1,2, \ldots, n-t$, we extract an element of the multiset $B=\{e_{b_1}, e_{b_2},\ldots, e_{b_{n-t}}\}$ and assign it to $f_{b_1}, f_{b_2},\ldots, f_{b_{n-t}}$ as follows: 
\[
f_{b_j} = {\rm max}\{k \ | \ k \in B - \{f_{b_1}, f_{b_2},\ldots, f_{b_{j-1}}\} \ {\rm and} \
k < {\rm max}(e_1, \ldots, e_{{b_j}-1}) \}.
\]
By definition, $f$ will satisfy the characterization property in Observation \ref{observation:201} of $\I_n(210)$.
\end{proof}

In order to get a recurrence that allows us to compute a few terms of $|\I_n(210)|=|\I_n(201)|$, we define two statistics, $\tp$ and $\bottom$, on $e \in \I_n(210)$ based on the decomposition of $e$ described in Observation \ref{observation:210}.
Let $e_{a_1} \leq e_{a_2} \leq \cdots \leq e_{a_t}$ be the sequence of  weak left-to-right maxima of $e$ and let
$e_{b_1} \leq e_{b_2} \leq \cdots \leq e_{b_{n-t}}$ be the subsequence of remaining elements of $e$. Then
\[
\tp (e) = e_{a_t}; \ \ \ \bottom (e) = e_{b_{n-t}}.
\]
If $e$ is weakly increasing, then $t=n$ 
so we define $\bottom (e)=-1$.

\begin{theorem}
Let $T_{n,a,b}$ be the number of $e \in \I_n(201)$ with $\tp (e) = a$ and  $\bottom (e)= b$.
Then
\begin{equation}
T_{n,a,b} = \sum_{i=-1}^{b}T_{n-1,a,i} +  \sum_{j=b+1}^{a}T_{n-1,j,b},
\label{Tnab:recurrence}
\end{equation}
with initial conditions 
$T_{n,a,b}=0$ if $a \geq n$ and $T_{n,a,-1}= \frac{n-a}{n}{n-1+a \choose a}$.
\label{rec:210ab}
\end{theorem}
\begin{proof}
 $T_{n,a,-1}$  is the number of weakly increasing inversion sequences with $e_n=a$.  This is the number of Dyck paths whose last horizontal step is at height $a$ which is $\frac{n-a}{n}{n-1+a \choose a}$.
 For $b\geq 0$, an inversion sequence  $e$ of length $n$  with $\tp (e) = a$ and $ \bottom (e) = b$ can be obtained 
 by adding $b$ to an $e'$ of length $n-1$ with $\tp (e') = a$ and $\bottom (e') \leq b$; or by adding $a$ to an $e'$ of length $n-1$ with $b < \tp (e')  \leq a$ and $\bottom (e') = b$.
\end{proof}

From Theorems \ref{210eq201} and \ref{rec:210ab} we have
\begin{equation}
|\I_n(210)|\ =\ |\I_n(201)|\  = \ \sum_{a=0}^{n-1} \sum_{b=-1}^{a-1} T_{n,a,b} \ =\  \frac{1}{n+1}\binom{2n}{n}  +\sum_{a=0}^{n-1} \sum_{b=0}^{a-1} T_{n,a,b}.
\label{count210}
\end{equation}
The first 12 terms of the sequence  $|\I_n(210)|$ are
\[
1, 2, 6, 24, 118, 674, 4306, 29990, 223668, 1763468, 14558588, 124938648.
\]
This sequence did not appear in the OEIS,  but has now been assigned A263777.

A different recurrence to compute $|\I_n(210)|=|\I_n(201|$ is derived in \cite{Mansour}.  It is more complicated than \eqref{Tnab:recurrence} due to the choice of statistics.  Nevertheless it is used to produce a generating function.
Can  \eqref{Tnab:recurrence}  be used to derive a simpler generating function?

\subsection{Inversion sequences avoiding {\bf 102}}

Our calculations showed that the number of inversion sequences in $\I_n$ avoiding 102  is:
$$1,2,6,22,89,381,1694,7744, 36168, \ldots . $$
We checked for $n \leq 9$ that 
this matches the sequence A200753 in the OEIS \cite{Sloane},  whose generating function is given by
\begin{equation}
 A(x)=1+(x-x^2)(A(x))^3
 \label{102gf}
 \end{equation}
 but we were unable to prove this.

In \cite{Mansour}, Mansour and Shattuck used a delicate generating function argument to confirm that
the generating function $\sum_{n \geq 0} |\I_n(102)|x^n$ does satisfy \eqref{102gf} and from that they provide an explicit formula for $|I_n(102)|$.

Is there a direct combinatorial argument to show that the generating function for the $102$-avoiding inversion sequences satisfies \eqref{102gf}?


\subsection{Avoiding {\bf 120}}

Our calculations show that the number of inversion sequences avoiding the pattern $120$ is given by:
\[
1,2,6,23,103,515,2803,16334, 100700, \ldots ,
\]
but this sequence did not appear in the OEIS (it has now been assigned A263778) and we do not yet know how to count it.  This remains an open question.

\section{Avoiding patterns with repeated letters}

\subsection{Avoiding {\bf 000}:  the Euler  numbers and simsun permutations}

The Euler up/down number $E_n$ is the number of permutations $\pi$ of $[n]$ such that $\pi_1<\pi_2 > \pi_3 < \pi_4> \cdots$.
This is a well-known interpretation of the coefficients in the Taylor series expansion of ${\rm tan}(x) + {\rm sec}(x)$:
\[
{\rm tan}(x) + {\rm sec}(x) = \sum_{n \geq 0} E_n \frac{x^n}{n!}.
\]

Several families are known to be in bijection with the up/down permutations, including 0-1-2-increasing trees \cite{KPP}.
These are $n$-vertex rooted unordered trees in which each vertex has at most two children. Each vertex has a distinct label from  the set $[n]$ and labels increase along any path from the root to a leaf.

\begin{figure}

\centering
\begin{tikzpicture}[every node/.style={circle,draw},level distance=1.2cm, sibling distance=2.2cm] 
\tikzstyle{level 3}=[level distance=1.2cm, sibling distance=1.2cm]
  \node {0} 
    child {node {1}
    	child {node {2}
		child {node {5}}
		child {node {7}}}
	child {node{6}
		child {node{9}}
		child[missing]{}}}
    child {node {3}
      child {node {4}
      	child {node {8}}
	child {node[inner sep=2pt] {10}} }}; 
\end{tikzpicture} \\

\caption{The tree corresponding to  $e=(0,1,0,3,2,1,2,4,6,4)$ under the bijection of Theorem \ref{thm:000}}
\label{SylvieTree}
\end{figure}
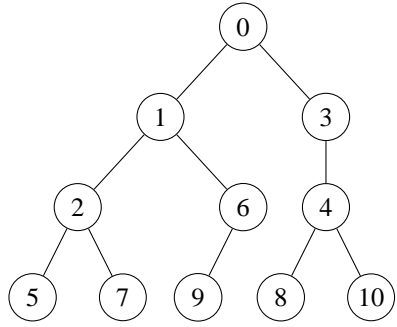

\begin{theorem}
$|\I_n(000)|= |E_{n+1}|$
\label{thm:000}
\end{theorem}
\begin{proof}
Observe that $e \in \I_n$ avoids $000$ if and only if no entry occurs more than twice.
We consider 0-1-2-increasing trees $T$ with $n+1$ vertices labeled $0,1, \ldots, n$, which are counted by $E_{n+1}$.
It is easy to check that the mapping sending such a tree $T$ to the inversion sequence $e$, where $e_i$ is the parent of $i$ in $T,$ is a bijection between these trees and $\I_n(000)$.  
\end{proof}

For an example of this bijection, see Figure \ref{SylvieTree}.  The bijection of  Theorem \ref{thm:000} shows that the set of labels on the internal vertices of $T$ is exactly the set
 $\{e_1, \ldots, e_n\}$. 
 \begin{theorem}
 Let $E_{n,k}$ be the number of $000$-avoiding inversion sequences of length $n$ with $k$ distinct entries.
 Then $$E_{n,k}=(n-k+1)E_{n-1,k-1}+(2k-n+1)E_{n-1,k}$$ with initial conditions $E_{0,0}=1$ and $E_{n,k}=0$ for $k>n$ or $k < \lceil(n/2)\rceil$.
 \end{theorem}
 \begin{proof}
 Elements of $\I_n(000)$ with $k$ distinct entries can be constructed either by appending an unused entry to end of an $e \in \I_{n-1}(000)$ with $k-1$ distinct entries
 or by appending a used (but unrepeated) entry to the end of an $e \in \I_{n-1}(000)$ with $k$ distinct entries.
 
 If $e \in \I_{n-1}(000)$ has $k-1$ distinct entries, then $n-k$ of the possible entries are unused and available for $e_n$.  Additionally, since $e_n=n-1$ is also possible, there are a total of $n-k+1$ choices.
 
 If $e \in \I_{n-1}(000)$ has $k$ distinct entries, then the other $n-1-k$ are repeats.  To append a used entry to $e$, while avoiding 000,  $e_n$ must be one of the $k-(n-1-k)$ used, but unrepeated, elements.  This gives a total of $2k-n+1$ choices.
 The recurrence follows.
 \end{proof}

 Another family counted by the Euler up-down numbers is the set $RS_n$ of {\em simsun permutations} \cite{sundaram}.
 A permutation is simsun if it has no double descents, even after removing $n, n-1, \ldots, k$ for any $k$.
 For example, 25637814 is not simsun: removing $8,7,6$ yields the permutation 25314, where $5<3<1$ is a double descent.  It is known (e.g. \cite{ChowShiu}) that if $rs_{n,k}$ is the number of simsun permutations of $[n]$ with $k$ descents then
 \begin{equation}
 rs_{n,k} = (k+1)rs_{n-1,k}+(n-2k+1)rs_{n-1,k-1}
 \label{simsundes}
 \end{equation}
 with initial conditions $rs_{0,0}=1$ and $rs_{n,k}=0$ for $k > \lfloor n/2 \rfloor$.

 We have the following relationship between simsun permutations and 000-avoiding inversion sequences.
 
 \begin{corollary}
 The number of $000$-avoiding inversion sequences in $\I_n$ with $n-k$ distinct entries is the
 number of simsun permutations of $[n]$ with $k$ descents.
 \label{000simsundes}
   \end{corollary}
 \begin{proof}
 The number of $e \in \I_n(000)$ with $n-k$ distinct entries is obtained by
 replacing $k$ by $n-k$ in the recurrence of the previous theorem.  Let $F_{n,k}=E_{n,n-k}$ be the number of $000$-avoiding inversion sequences in $\I_n$ with $n-k$ distinct entries.  This gives:  $$F_{n,k}=(n-(n-k)+1)F_{n-1,k}+(2(n-k)-n+1)F_{n-1,k-1}=(k+1)F_{n-1,k}+(n-2k+1)F_{n-1,k-1}$$ with initial conditions $F_{0,0}=1$ and $F_{n,k}=0$ for $k >\lfloor(n/2)\rfloor$, the same recurrence satisfied by $rs_{n,k}$.
 \end{proof}

It would be interesting to  have a natural bijection for Corollary \ref{000simsundes}.

The {\em Entringer numbers}, $d_{n,k}$, count the number of  down/up permutations of $[n+1]$ with first entry equal to $k+1$.
(These are A008281 in the OEIS \cite{Sloane}.)
They satisfy the recurrence $$d_{n,k}=d_{n,k-1}+d_{n-1,n-k}.$$
Our calculations suggest the following.

\begin{conjecture}
For $n \geq 1$ and $0 \leq k \leq n-1$, $d_{n,k}$  is the number of  $e \in \I_n(000)$ with $e_n=k-1$.
\end{conjecture}

\subsection{Avoiding {\bf 001}:  $2^{n-1}$ and $\Sn_n(132,231)$}

In this section, we show that the 001-avoiding inversion sequences are counted by powers of 2.  This indicates a natural connection between $\I_n(001)$ and permutations of length $n$ avoiding certain patterns of length~3: the permutations in $\Sn_n$ avoiding both $213$ and $312$ are counted by $2^{n-1}$, as are nine other pairs of permutations of $123$. This was shown by Simion and Schmidt in \cite{SimionSchmidt}.
Rotem in \cite{Rotem} showed the $(231,312)$ case.

\begin{theorem}[Simion-Schmidt]
$|\Sn_n(\alpha, \beta)|=2^{n-1}$
for any of the following pairs  $(\alpha, \beta)$ of patterns:
\[
(123,132),
(123,213),
(132,213),(132,231),(132,312),
\]
\[
(213,231),
(213,312),(231,312),
(231,321),(312,321).
\]
\label{SimionSchmidt}
\end{theorem}

\begin{observation} For $n \geq 1$, $\I_n(001)$ is the set of $e \in \I_n$  
 satisfying, for some $t \in [n]$,
\begin{equation}
e_1 < e_2 < \cdots < e_t \geq e_{t+1} \geq e_{t+2} \geq \cdots \geq e_n.
\label{equation:64A}
\end{equation}
\label{observation:64A}
\end{observation}

\begin{theorem}
For $n \geq 1$, $|\I_n(001)|=2^{n-1}$.
\label{theorem:64A}
\end{theorem}
\begin{proof}  We give two proofs based on Observation \ref{observation:64A}.

{\em First proof.}  In an inversion sequence  $e$ satisfying \eqref{equation:64A} for some $t$, it must be the case that $e_i$ is maximal (i.e. $e_i=i-1$) whenever $1 \leq i \leq t$.  It follows that the rest of the sequence $p=(e_{t+1}, \ldots, e_n)$ can be viewed as a partition that fits in an $n-t$ by $t-1$ box, of which there are $\binom{(n-t)+(t-1)}{t-1} = \binom{n-1}{t-1}$.  Summing over $t$ from $1$ to $n$ gives the result.

{\em Second proof.}  
Recall from Section 1 the bijection  $\Theta(\pi):  \Sn_n \rightarrow \I_n$ for  $\pi=\pi_1 \ldots \pi_n \in \Sn_n$  defined by $\Theta(\pi)=(e_1,e_2, \ldots, e_n)$, where $e_i = |\{j \ | \ j < i \ {\rm and} \ e_j > e_i \}|$.  Note that $e \in \I_n$ satisfies \eqref{equation:64A}
if and only if
 $\pi = \Theta^{-1}(e)$ satisfies
\[
\pi_1 > \pi_2 > \cdots > \pi_t < \pi_{t+1}  < \pi_{t+2} < \cdots < \pi_n.
\]
Such permutations are the ones that avoid both $132$ and $231$, so the result follows by Theorem \ref{SimionSchmidt}.
\end{proof}

\subsection{Avoiding {\bf 011}: the Bell numbers}

In this section we show that the $011$-avoiding inversion sequences are counted by the Bell numbers.

The {\em Bell number} $B_n$ is the number of ways to partition an $n$-element set into nonempty subsets called {\em blocks}.
The numbers $S_{n,k}$ of such  set partitions into $k$ blocks are known as the {\em Stirling numbers of the second kind} and they satisfy the recurrence $S_{n,k}=kS_{n-1,k}+S_{n-1,k-1}$ with initial conditions $S_{n,1}=S_{n,n}=1$.

\begin{observation}
The $011$-avoiding inversion sequences are those in which the positive entries are distinct.
\label{observation:877A}
\end{observation}

\begin{theorem}
The number of $011$-avoiding inversion sequences in $\I_n$ with $k$ zeros is $S_{n,k}$.
\label{thm:011}
\end{theorem}
\begin{proof}
An inversion sequence $e \in \I_n(011)$ with $k$ zeros can arise from an inversion sequence in $\I_{n-1}(011)$ in one of two ways.
If $e_n=0$ then $(e_1, \ldots, e_{n-1}) \in \I_{n-1}(011)$ has $k-1$ zeros.
Otherwise, $(e_1, \ldots, e_{n-1}) \in \I_{n-1}$ has $k$ zeros and, by Observation \ref{observation:877A},  it has $n-1-k$ distinct positive entries (out of the possible $n-2$ positive elements of  $[n-2]$.)  This means that any of the remaining  $k-1$ positive elements can be assigned to $e_n$, as well as the new possibility $n-1$, for a total of $k$.  
Since the only $e \in \I_n(011)$ with one zero is $(0,1,2, \ldots, n)$ and 
the only $e \in \I_n(011)$ with $n$ zeros is $(0,0, \ldots, 0)$,
the recurrence of the Stirling numbers is satisfied with the same initial conditions.
\end{proof}

A {\em restricted growth function} is a finite integer sequence $v=(v_1, \ldots, v_n)$ with $v_1=1$ and for $1<i \leq n$,
$v_i \leq 1+ {\rm max}\{v_1, \ldots, v_{i-1}\}$.  Let $G_n$ be the set of restricted growth functions of length $n$.
Elements of $G_n$ encode partitions of $[n]$: given a partition $\Pi$ of $[n]$, order the blocks of $\Pi$ as
$B_1, \ldots, B_k$ so that min$(B_i)<$ min$(B_{i+1})$ 
for $1 \leq i < k$.  Then $v \in G_n$ corresponds to the set partition $\Pi$ 
where $i$ is in block $B_b$ of $\Pi$ if and only if $v_i=b$.  The number of distinct entries of $v$ is the number of blocks of $\Pi$.

For example if $v=(1,2,3,1,3,2,4,5,6,3,4,2)$ then \[\Pi = (\{1,4\},\ \{2,6,12\},\ \{3,5,10\},\  \{7,11\},\ \{8\}, \{9\}).\]

The proof of  Theorem \ref{thm:011} gives rise to a bijection from $\I_n(011)$ to $G_n$.
For an integer sequence $s$, let zeros$(s)$ be the number of zeros in $s$.
Let $\I(011)$ be the set of all $011$-avoiding inversion sequences, regardless of length and $G$ the set of all restricted growth sequences.

Define a map $\kappa: \I(011) \rightarrow G$ for $e \in \I(011)$ recursively.
If $|e|=1$, then $e=(0)$ and we define $\kappa(e)=(1)$.
For $|e|=n>1$, assume $\kappa(e_1, \ldots, e_{n-1})$ has been defined and let $k=$ zeros$(e)$.
Recall from the proof of Theorem \ref{thm:011} that since $e \in \I_n(011)$, if $e_n>0$ then $e_n$ must be one of the $k$ elements of
$[n-1] - \{e_1, \ldots e_{n-1}\}$, call them $a_1 < a_2 < \cdots < a_k$.
So with that notation, we define $\kappa(e) = \kappa(e_1, \ldots, e_{n-1}) \cdot v_n$
where $v_n = k+1$ if $e_n=0$ and $v_n = i$ if $0<e_n=a_i$.

For example, $\kappa(0,0,0,1,4,3,0,0,0,6,8,5)=(1,2,3,1,3,2,4,5,6,3,4,2)$

It is not hard to prove by induction that if 
$e =(e_1, \ldots, e_{n}) \in \I(011)$ and $\kappa(e_1, \ldots, e_{n}) = (v_1, \ldots, v_n)$ then
$1 \leq v_i \leq {\rm zeros}(e_1, \ldots, e_i)$ for all $i$, from which it clearly follows that $(v_1, \ldots, v_n) \in G$.  It is straightforward to reverse $\kappa$, giving a length-preserving bijection between $\I(011)$ and $G$.

\subsection{Wilf-equivalent patterns {\bf 101} and {\bf 110}:   $\Sn_n$(1-23-4)}

Our calculations suggested that the patterns $101$ and $100$ are Wilf equivalent on inversion sequences. The avoidance sequence for both patterns agree with sequence A113227 in the OEIS \cite{Sloane}, where it is said to be $|\Sn_n$(1-23-4)$|$, the number of permutations avoiding the pattern 1-23-4, that is, with no $i < j < k$ such that $\pi_i < \pi_j < \pi_{j+1} < \pi_k$. 
The asymptotics of $|\Sn_n$(1-23-4)$|$ were studied by Elizalde in \cite{elizalde}, where he established good upper and lower bounds.

The OEIS led us to a paper of David Callan \cite{callan10}, which shows that permutations of $[n]$ avoiding 1-23-4 are in bijection with
increasing ordered trees with $n+1$ vertices whose leaves, taken in preorder, are also increasing.  Callan showed that if $u_{n,k}$ is the number of such trees with $n+1$ vertices in which the root has $k$ children then
\begin{equation}
u_{n,k} = u_{n-1,k-1} + k \sum_{j=k}^{n-1} u_{n-1,j}
\label{callan}
\end{equation}
with initial conditions $u_{0,0}=1$ and $u_{n,k}=0$ if $k > n$, or $n>0$ and $k=0$.

\begin{theorem} \label{theorem:101}
$|\I_n(101)| = |\Sn_n$\rm{(1-23-4})$|$.
\end{theorem}
\begin{proof}
Let $Z_{n,k}$ be the number of $e \in \I_n(101)$ with exactly $k$ zeros and let $z_{n,k} = |Z_{n,k}|$.  We show that $z_{n,k}$
satisfies \eqref{callan} with the same initial conditions.

Let $Z_{n,k,\ell}$ be the number of $e \in Z_{n,k} $ with exactly $\ell$ ones.
Recall that applying $\sigma_{-1}$ to an inversion sequence decreases the positive entries by 1.

Define $\gamma:  \I_n \rightarrow \I_{n-1}$ by $\gamma(e_1, \ldots, e_n)= \sigma_{-1}(e_2, \ldots, e_n)$.
Note that $\gamma(Z_{n,k,\ell}) = Z_{n-1,k+\ell-1}$.
In fact, if $\ell=0$, $\gamma$ is a bijection between $Z_{n,k,0}$ and $ Z_{n-1,k-1}$.
However, if $\ell >0$, each element of $ Z_{n-1,k-1}$ is the image under $\gamma$ of $k$ elements of $Z_{n,k,\ell}$.
To see this, let $e \in Z_{n,k,\ell}$ and let ${\tilde e}=(e_{b_1}, e_{b_2}, \ldots, e_{b_{k+\ell}})$, 
be the subsequence of $e$ consisting of the zeros and ones in $e$.
Note $e_{b_1}=0$ and, since $e$ avoids $101$, the $\ell$ ones in ${\tilde e}$ must be consecutive.
There are  $k$ such ways to place the ones in ${\tilde e}$, namely as:
\[
(e_{b_2}, \ldots, e_{b_{1+\ell}}),
(e_{b_3}, \ldots, e_{b_{2+\ell}}), \ldots 
(e_{b_{k+1}}, \ldots, e_{b_{k+\ell}}).
\]
Thus, since $z_{n,k} 
=\sum_{\ell=0}^{n-k} |Z_{n,k,\ell}| = |Z_{n,k,0} |+ \sum_{\ell=1}^{n-k} |Z_{n,k,\ell}|$, we have

\[
z_{n,k} 
\ =\ |Z_{n-1,k-1}| + \sum_{\ell=1}^{n-k} k|Z_{n-1,k+\ell-1}|
\ = \ z_{n-1,k-1} + k\sum_{\ell=1}^{n-k} z_{n-1,k+\ell-1}.
\]
Clearly 
$z_{0,0}=1$ and $z_{n,k}=0$ if $k > n$, or $n>0$ and $k=0$.
Re-indexing the summation gives \eqref{callan}.
\end{proof}

Comparing the proof of the recurrence for $z_{n,k}$ here with Callan's proof of the recurrence for $u_{n,k}$ gives rise to a bijection between $101$-avoiding inversion sequences with $k$ zeros and ordered increasing trees with increasing leaves in which the root has $k$ children.  We omit the details for now.

\begin{theorem}
$|\I_n(101)| = |\I_n(110)|$.
\end{theorem}
\begin{proof}
We observe that the number of $e \in \I_n(110)$ with $k$ zeros and $\ell$ ones satisfies the same recurrence \eqref{callan}.
The proof follows the same process as that of Theorem \ref{theorem:101} to the point of considering the substring of 0's and 1's for some $e \in \I_n(110)$.  Since $e$ now avoids $110$, in the substring
${\tilde e}=(e_{b_1}, e_{b_2}, \ldots, e_{b_{k+\ell}})$, 
consisting of its zeros and ones,  all but the first $\ell$ ones must be in the last $\ell-1$ positions of
${\tilde e}$.  This leaves $k$ possible positions for the first one:
$b_2, b_3, \ldots, b_{k+1}$.
\end{proof}

Using the recursion \eqref{callan} for $ u_{n,k+1}$  and  $ u_{n,k}$, if we take $k u_{n,k+1}$ and subtract $(k+1)u_{n,k}$ the result is
\[
u_{n,k+1} + k u_{n-1,k}  = \frac{k+1}{k} \left ( u_{n,k} - u_{n-1,k-1} \right ) .
\]
Is there a simple interpretation for this in terms of inversion sequences, trees, or permutations?

\subsection{Avoiding {\bf 010}, avoiding {\bf 100 }}

Our calculations show that the number of inversion sequences avoiding the pattern $010$ is given by:
\[
1, 2, 5, 15, 53, 215, 979, 4922, 26992, \ldots ,
\]
but this sequence did not appear in the OEIS (is is now A263779) and we do not yet know how to count it.

Similarly, we have calculated that $|\I_n(100)|$ begins like this:
\[
1, 2, 6, 23, 106, 565, 3399, 22678, 165646, \ldots ,
\]
but it also did not appear in the OEIS (it is now A263780).

\section{Revisiting 021-avoidance}

In this section, we delve deeper into the structure of 021-avoiding inversion sequences and their relationship to  structures counted by the Schr\"oder numbers.  First, we introduce a bijection that illustrates further natural correspondences between statistics in inversions sequences and Schr\"oder paths.  After, we consider the ascent statistic, which  we show is symmetric in $\I_n(021)$ via a bijection with certain black/white binary trees.

\subsection{Another correspondence between $\I_n(021)$ and Schr\"oder paths}


Now we define a different bijection, which serves as a tool to relate different statistics between inversion sequences and Schr\"oder paths.  Using Observation \ref{observation:021}, we can see that for any $e \in \I_n(021)$, $e$ can be written as $e=\bb_0\bb_1\bb_2\ldots\bb_{n-1}$ where $\bb_k$ is the substring $(e_i,e_{i+1},\ldots,e_j)$ such that $e_i$ is the first occurrence of $k$, and, for every $t \in \{i, i+1, \ldots , j\}$, we have $e_t=k$ or $e_t=0$.  We call each $\bb_k$ a \emph{block} and say that some $\bb_k$ is \emph{maximal} if $\bb_k=(e_{k+1},e_{k+2},\ldots,e_j)$ (making $e_{k+1}$ a maximal entry).  Notice that any maximal entry $e_{j+1}$ must be the first entry of $\bb_j$.

For example, if $e=(0,1,0,1,0,2,5,7,7,7,9,0,10,11,12)$, then $\bb_0=(0)$, $\bb_1=(1,0,1,0)$, $\bb_2=(2)$, $\bb_5=(5)$, $\bb_7=(7,7,7)$, $\bb_9=(9,0)$, $\bb_{10}=(10)$, $\bb_{11}=(11)$, $\bb_{12}=(12)$, and $\bb_3=\bb_4=\bb_6=\bb_8=\bb_{13}=\bb_{14}$ are all empty strings.  Additionally, $\bb_0$, $\bb_1$ and $\bb_{7}$ are maximal.

This decomposition is essential to describing our bijection.  It also produces a recurrence relation on 021-avoiding inversion sequences.

\begin{theorem}
Let $Y_{n,k}$ be the number of $021$-avoiding inversion sequences of length $n$ with $k$ maximal elements, including $e_1=0$.  Then
\begin{equation}
Y_{n,k} = Y_{n-1,k-1} + 2 \sum_{i=k}^{n-1} Y_{n-1,i},
\end{equation}
with initial condition $Y_{1,1}=1$.
\end{theorem}

\begin{proof}
Given a 021-avoiding inversion sequence of length $n$ with $k$ maximal elements, let $\bb_j$ be the last maximal block.  Consider the effect of removing the last entry of $\bb_j$.  If $\bb_j$ contains only one element, then $\bb_j=\bb_n$  since it is the last maximal block, and removing $\bb_n$ gives an inversion sequence contributing to $Y_{n-1,k-1}$.  If $\bb_j$ contains more than one element, it could end in $j$ or $0$.  Either way, removing the last entry shifts everything following $\bb_j$ one position earlier.  This shift may result in a number of additional maximal blocks following $\bb_j$, so the resulting inversion sequence contributes to the term $Y_{n-1,i}$ where $k \leq i \leq n-1$.

Conversely, a 021-avoiding inversion sequence of length $n$ with $k$ maximal elements can be obtained from a 021-avoiding inversion sequence of length $n-1$ with $k-1$ maximal elements by appending $n-1$ to the end.  Additionally, given any 021-avoiding inversion sequence of length $n-1$ with at least $k$ maximal elements, if $\bb_j$ is the $k$th maximal block, we can obtain a some $e \in \I_n(021)$ with $k$ maximal blocks by appending either $j$ or $0$ to the end of $\bb_j$.  By doing this, all maximal blocks following $\bb_j$ are no longer maximal, since the entries have been shifted, resulting in the desired number of maximal blocks.
\end{proof}

A \emph{valley} in a Schr\"oder Path is a $D$ step immediately followed by an $U$ step.  Let the \emph{valley word} of a Schr\"oder Path be the word obtained when any consecutive $DU$ is replaced with a $V$.  So the Schr\"oder path $UUDUFUDUFDDDUUDUDDUUUFDDD$  from Figure \ref{Schroeder Path} would have valley word $$UUVFUVFDDVUVDVUUFDDD.$$ We define a mapping $\phi: R_{n-1} \longrightarrow \I_n(021)$ using the valley word for each element in $R_{n-1}$.  The entries of the valley word are interpreted as instructions for building an inversion sequence.

Let $p$ be the valley word of a path in $R_{n-1}$.  Define $M$ to be a word on the elements $\{0,1,2,\ldots,n-1\}$ that keeps track of the current maximal blocks in the inversion sequence being built, which we denote $e$.  Initially, set $M=0$ and $e=\bb_0$ where $\bb_0=(0)$.  So, the initial length of $e$ is 1.

The valley word $p=p_1p_2\ldots p_{\ell}$ is read left to right and each $p_i$ interpreted as an action performed on $M$ and $e$.
\begin{itemize}
\item If $p_i=U$, append $\ell_i$ to the end of $e$ and $M$, where $\ell_i$ is equal to the current length of $e$.  This is equivalent to starting a new block $\bb_{\ell_i}$ at the end of $e$.

\item If $p_i=D$, delete the last entry of $M$.

\item If $p_i=V$, append a 0 to the end of block $\bb_j$ where $j$ is the last entry of $M$.

\item If $p_i=F$, append a $j$ to the end of block $\bb_j$ where $j$ is the last entry of $M$.
\end{itemize}

Notice that this construction yields an inversion sequence with weakly increasing positive entries and therefore avoids 021.  Additionally, $\phi$ is reversible and is a bijection.  As an example, \[\phi(UUV FUV FDDV UV DV UUFDDD)=(0,1,0,0,2,0,2,5,0,5,9,0,12,13,13),\] is the image of the Schr\"oder path in Figure \ref{Schroeder Path}.
Notice that this is not the same as \[\rho(UUV FUV FDDV UV DV UUFDDD)\] from Section 2.1.

The bijection $\phi$ succeeds in relating different statistics in 021-avoiding inversion sequences and Schr\"oder paths.  In the following theorem, the number of \emph{late zeros} in an inversion sequence $e \in \I_n(021)$ is the number of zeros occurring in the blocks $\bb_1 \bb_2 \ldots \bb_{n-1}$.  Additionally, the number of distinct nonzero values in an inversion sequence $e=\bb_0\bb_1 \bb_2 \ldots \bb_{n-1} \in \I_n(021)$ is $|\{\bb_i \mid i \neq 0, \bb_i \neq \epsilon\}|$.  By using the definition of $\phi$, we achieve the following results.

\begin{corollary} \label{cor:phi}
For $n \geq 1$,
\begin{enumerate}
\item if $p\in R_{n-1}$ with $k$ valleys, then $\phi(p)$ is an inversion sequence in $\I_n(021)$ with $k$ late zeros.  
\item if $p\in R_{n-1}$ where $k$ is the number of occurrences of $U$ in the valley word of $p$ (alternatively, this is the total number of up steps minus the number of valleys), then $\phi(p)$ is an inversion sequence  in $\I_n(021)$ with $k$ distinct nonzero values.
\item if $p\in R_{n-1}$ has $k$ flat steps at height 0, then $\phi(p) \in \I_n(021)$  begins with $k+1$ zeros.
\end{enumerate}
\end{corollary}

The number of Schr\"oder paths of length $n$ with $k$ valleys is counted by A101282 in OEIS.

Based on our computations, the following statistic seems to correspond between inversion sequences and Schr\"oder paths.  However, neither $\rho$ nor $\phi$ provide the necessary correspondence.  An \emph{ascent} in a Schr\"oder path is a maximal sequence of consecutive up steps.

 \begin{conjecture} \label{conj:021 ascents}
 The number of $p \in R_{n-1}$ with $k-1$ ascents is equal to the number of $e \in \I_n(021)$ with $k$ distinct  values. 
 \end{conjecture}
 
 The sequence in Conjecture~\ref{conj:021 ascents} is counted by A090981 is OEIS.  Additionally, notice that Conjecture \ref{conj:021 ascents} in tandem with (2) from Corollary~\ref{cor:phi} would prove that the number of Schr\"oder paths in $R_{n-1}$ with $k$ ascents is equal to the number of paths in $R_{n-1}$ with $k$ occurrences of $U$ that are not a part of a valley, $DU$.

\subsection{A symmetric statistic on 021-avoiding inversion sequences}

As {\em ascent} in an inversion sequence $e$ is an index $i$ such that $e_i < e_{i+1}$; the number of ascents in $e$ is denoted by $\asc(e)$.  In this section, we show that the ascent statistic is symmetric on $021$-avoiding inversion sequences.  That is, if $a_{n,k}$ is the number of $e \in \I_n(021)$ with $k$ ascents then $a_{n,k}=a_{n,n-k-1}$.

We make use of a tree structure that appears in the thesis of Brian Drake \cite{drake}.
  Define $\TT_{n-1}$ to be the set of rooted binary trees on $n-1$ nodes, where each node is either black or white, and no node is the same color as its right child.  For an example, see Figure \ref{funny tree}.  Let $\TT$ be the set of all such trees with no restriction on the number of nodes.    
In \cite{drake}, Drake uses an inversion theorem for labeled trees to compute the generating function for $\TT$, keeping track of the number of black nodes.  One consequence is that 
$|\TT_n|=r_n$, the $n$th large Schr\"oder number.

The trees in $\TT_{n}$ are also related by a natural bijection to the separable permutations in $\Sn_n$.
A \emph{separable permutation} is a permutation that can be completely decomposed with direct and skew sums; the trees in $\TT_n$ provide the recipe for this decomposition.  (See  \cite{Albert}.)
The separable permutations in $\Sn_n$ are exactly those that avoid $2413$ and $3142$ and it is known that $|\Sn_n(2413,3142)|=r_n$ (\cite{shapiro,west}).

Returning to the ascent statistic on inversions sequences, observe that for a fixed number of nodes, the number of trees in $\TT$ with $k$ black nodes is the same as the number with $k$ white nodes.  Thus the symmetry of the ascent statistic on $\I_n(021)$ is a consequence of the following, which we will prove.

\begin{theorem}
The number of trees in $\TT_{n-1}$ with $k$ black nodes is the same as the number of inversion sequences in $\I_n(021)$ with $k$ ascents.
\label{thm:symm021}
\end{theorem}

Theorem \ref{thm:symm021} will follow from Proposition \ref{funny} below once we define an appropriate bijection. From that we will have the following.

\begin{corollary}
Let $a_{n,k}$ be the number of $e \in \I_n(021)$ with $k$ ascents.
 Then $a_{n,k}=a_{n,n-k-1}$.  
\label{ascsym}
\end{corollary}

We define a bijection between $\I_n(021)$ and $\TT_{n-1}$ such that the number of ascents in an inversion sequence is equal to the number of black nodes in the corresponding tree.  Define two operations on $\TT$.  Given $T, S \in \TT$, let $\omega(T,S)$ be the tree with a white root that has left subtree $T$ and right subtree $S$.  Similarly, define $\beta(T,S)$ to be the tree with a black root that has left subtree $T$ and right subtree $S$.  Note that in $\omega(T,S)$ and $\beta(T,S)$, the tree $S$ is required to have a black root and white root, respectively.

Throughout this section, let $\I^{00}_n(021)=\{e \in \I_n(021) \mid e_2=0 \}$ and $\I_n^{01}=\{e \in \I_n(021) \mid e_2=1\}$.  These sets are nonempty only for $n \geq 2$.

We define a bijection $\tau: \I_n(021) \rightarrow \TT_{n-1}$ recursively.  Set $\tau(0)$ to be the empty tree.  The mapping $\tau$ will be defined such that any inversion sequence in $\I_n^{00}(021)$ maps to a tree with a white root and any inversion sequence in $\I_n^{01}(021)$ maps to a tree with a black root.  Anytime we consider the operations $\omega(\tau(e'), \tau(e''))$ and $\beta(\tau(e'), \tau(e''))$, we will ensure that $e''$ begins with $0,1$ and $0,0$, respectively, or $e''=(0)$ in order to satisfy the condition on right children.

Given any inversion sequence $e=(e_1,e_2,e_3, \ldots, e_n) \in \I_n(021)$, we can consider two cases based on whether $e_2=0$ or $e_2=1$ (which is equivalent to $e \in \I_n^{00}(021)$ or $e \in \I_n^{01}(021)$, respectively).  Let $\ell$ be the largest value such that $e_2=e_3=\ldots=e_{\ell+1}$.  Additionally, let $k+1$ be the earliest position after $\ell+1$ such that $e_{k+1} \geq k-\ell+1$.   If there is no such position, set $k=n$.  Define \[
\tau(e)=
\begin{cases}
\omega(\tau( 0^{\ell}\cdot  \sigma_{\ell-k}(e_{k+1},e_{k+2},\ldots, e_n)), \tau(0,e_{\ell+2},e_{\ell+3},\ldots, e_k))&\mbox{if }e_2=0 \\ 
\beta(\tau( 0^{\ell} \cdot  \sigma_{\ell-k}(e_{k+1},e_{k+2},\ldots, e_n)), \tau(0,e_{\ell+2},e_{\ell+3},\ldots, e_k))&\mbox{if }e_2=1 \\
\end{cases} .
\]
where $0^{\ell}$ denotes the sequence of $\ell$ zeros.
For an example, see Figure \ref{funny tree}. 

First we show that  $0^{\ell} \cdot \sigma_{\ell-k}(e_{k+1},e_{k+2},\ldots, e_n)$ is a 021-avoiding inversion sequence.  
For $0^{\ell} \cdot  \sigma_{\ell-k}(e_{k+1},e_{k+2},\allowbreak \ldots, e_n)$ each entry $e_{k+i}$ is in the $(\ell+i)$-th position.  So we must show that whenever $e_{k+i} \neq 0$, we have $1 \leq e_{k+i}-k+\ell < \ell+i$.  By choice of $k$ and Observation \ref{observation:021}, it follows that $k-\ell+1 \leq e_{k+i}< k+i$ and the result immediately follows.  Finally, note that $0^{\ell}\cdot  \sigma_{\ell-k}(e_{k+1},e_{k+2},\ldots, e_n)$ will avoid the pattern 021, since its nonzero entries must weakly increase.

Now consider $(0,e_{\ell+2},e_{\ell+3},\ldots, e_k)$.  We know the nonzero entries of $(0,e_{\ell+2},e_{\ell+3},\ldots, e_k)$ will weakly increase.  Additionally, for $i\in\{\ell+2,\ell+3,\ldots,k\}$, $e_{i}$ is in position $i-\ell$.  Whenever $e_i \neq 0$ we want $1 \leq e_i< i-\ell$; this immediately follows by choice of $k$.  Therefore, $(0,e_{\ell+2},e_{\ell+3},\ldots, e_k)$ is a 021-avoiding inversion sequence. Notice that the definitions of $\ell$ and $k$ imply that $e_{\ell+2}=1$ when $e_2=0$ and $e_{\ell+2}=0$ when $e_2=1$, which is necessary to satisfy the condition on right children.

The mapping $\tau$ has an inverse, ensuring that it is a bijection from $\I_{n}(021)$ to $\TT_{n-1}$.

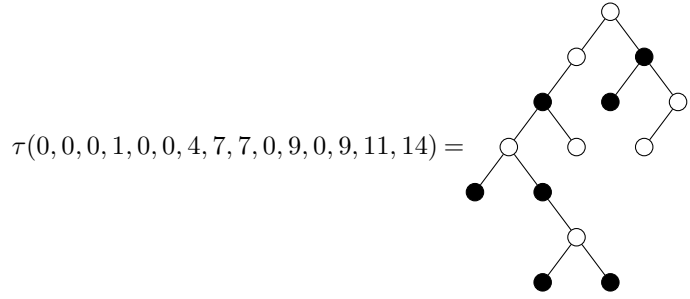
\begin{figure}
\centering
\tikzset{dot/.style={draw,shape=circle,fill=black,scale=.7}}
\tikzset{wdot/.style={draw,shape=circle, fill=white, scale=.7}}
\begin{tikzpicture}

\draw(-4,0)node[]{$\tau(0,0,0,1,0,0,4,7,7,0,9,0,9,11,14)=$};

\draw (.5,0)node[]{
\begin{tikzpicture}[level distance=.6cm, sibling distance=.9cm]
\node[wdot]{}
  child{node[wdot]{} child{node[dot]{} child{node[wdot]{} child{node[dot]{}} child{node[dot]{} child[missing]{} child{node[wdot]{} child{node[dot]{}} child{node[dot]{}}}}} child{node[wdot]{}}} child[missing]{}}
  child{node[dot]{} child{node[dot]{}} child{node[wdot]{}  child{node[wdot]{}} child[missing]{}}};
\end{tikzpicture}
};

\end{tikzpicture}
\caption{An example of $\tau:\I_n(021) \rightarrow \TT_{n-1}$.} \label{funny tree}
\end{figure}

Our bijection provides a correspondence between a number of statistics on inversion sequences and trees. 
First, we settle Theorem \ref{thm:symm021} with the following proposition. 
 Given $T \in \TT$, define $B(T)$ to be the number of black nodes in $\TT$.

\begin{proposition}
For any $e \in \I_n(021)$, $\asc(e)=B(\tau(e))$.
\label{funny}
\end{proposition}
\begin{proof} (Induction.)
We immediately see that this is true for $n=1$.  For $n>1$, let $e \in \I_{n}(021)$ with $j$ ascents.  We show that $\tau(e)$ has $j$ black nodes.  Set $\ell$ and $k$ to be defined as in the definition of $\tau$.

If $\ell=n-1$, then $e=(0,0,\ldots,0)$ or $e=(0,1,1,\ldots,1)$.  In either of these cases, the only nodes of $\tau(e)$ are in the leftmost branch.  Additionally, $\tau(0,0,\ldots,0)$ is a tree with all white nodes and $\tau(0,1,1,\ldots,1)$ is a tree whose root is black and all other nodes are white.  So, $B(\tau(e))=\asc(e)$ in either case.

If $\ell<n-1$ and $k=n$, then $\ell+1$ is an ascent; additionally, if $e \in \I_n^{01}(021)$, then $1$ is guaranteed to be an ascent. The only ascents of $e$ besides these are elements of $\{\ell+2,\ell+3,\ldots, n\}$.  Notice that $\tau(e)=\omega(\tau( 0^{\ell}), \tau(0,e_{\ell+2},e_{\ell+3},\ldots, e_n))$ if $e \in \I_n^{00}(021)$ or $\tau(e)=\beta(\tau( 0^{\ell}), \tau(0,e_{\ell+2},e_{\ell+3},\ldots, e_n))$ if $e \in \I_n^{01}(021)$.  In each of these, $B(\tau(e))=B(\tau(0,e_{\ell+2},e_{\ell+3},\ldots, e_n))$ and $B(\tau(e))=1+\tau(0,e_{\ell+2},e_{\ell+3},\allowbreak\ldots, e_n)$ respectively.  It follows inductively that $B(\tau(e))=\asc(e)$.

Now, if $\ell<n-1$ and $k<n$, we have the following argument, which we break into two cases.
If $e_2=0$, then the choice of $\ell$ and $k$ requires that $\ell+1$ and $k$ are ascent positions.  Additionally, there are no ascents in $(e_1,e_2,e_3,\ldots, e_{\ell+1})$, so $\asc(e)=1+\asc(e_{\ell+2},e_{\ell+3},\ldots, e_k)+1+\asc(e_{k+1},e_{k+2},\ldots, e_n)$.  It follows that
\[\begin{array}{rcl}
B(\tau(e))&=&B(\tau(0^{\ell} \cdot \sigma_{\ell-k}(e_{k+1},e_{k+2},\ldots, e_n)))+B(\tau(0,e_{\ell+2},e_{\ell+3},\ldots, e_k))\\
&=&\asc(0^{\ell} \cdot \sigma_{\ell-k}(e_{k+1},e_{k+2},\ldots, e_n))+ \asc(0,e_{\ell+2},e_{\ell+3},\ldots, e_k)\\
&=&\asc(e_{k+1},e_{k+2},\ldots, e_n)+1+ \asc(e_{\ell+2},e_{\ell+3},\ldots e_k)+1\\
&=&\asc(e). \\
\end{array} \]
If $e_2=1$, then the definitions of $\ell$ and $k$ require that $1$, $\ell+1$, and $k$ are ascent positions.  Additionally, there are no ascents in $(e_2,e_3,\ldots, e_{\ell+1})$, so $\asc(e)=1+1+\asc(e_{\ell+2},e_{\ell+3},\ldots, e_k)+1+\asc(e_{k+1},e_{k+2},\ldots, e_n)$.  It follows that
\[\begin{array}{rcl}
B(\tau(e))&=&1+B(\tau(0^{\ell} \cdot  \sigma_{\ell-k}(e_{k+1},e_{k+2},\ldots, e_n)))+B(\tau(0,e_{\ell+2},e_{\ell+3},\ldots, e_k))\\
&=&1+\asc(0^{\ell} \cdot \sigma_{\ell-k}(e_{k+1},e_{k+2},\ldots, e_n))+ \asc(0,e_{\ell+2},e_{\ell+3},\ldots e_k)\\
&=&1+\asc(e_{k+1},e_{k+2},\ldots, e_n)+1+ \asc(e_{\ell+2},e_{\ell+3},\ldots, e_k)+1\\
&=&\asc(e). 
\end{array} \]
\end{proof}

It would be nice to have a direct combinatorial proof of Corollary \ref{ascsym}.

There are two other natural statistics on trees and inversion sequences for which $\tau$ provides a correspondence.\begin{theorem}
The number of maximal values in an inversion sequence $e \in \I_n(021)$, not counting the initial zero of $e$, is equal to the number of black nodes in the leftmost branch in $\tau(e)$.
\end{theorem}
\begin{proof}
Let $\e \in \I_n(021)$ with $t$ maximal values.  By choice of $\ell$ and $k$ in the definition of $\tau$, none of the entries  $e_3,e_4, \ldots, e_k$  of $e$ can be maximal.  Therefore, the maximal values of $e$ will all occur among the entries $e_1,e_2,e_{k+1},e_{k+2},\ldots, e_n$.  Additionally, if $e_i$ is maximal, where $i \in \{k+1,k+2,\ldots,n\}$, then that corresponding entry will be maximal in $0^{\ell} \cdot \sigma_{\ell-k}(e_{k+1},e_{k+2},\ldots, e_n)$.

The number of black nodes in the left branch of $\tau(e)$ is equal to the number of black nodes in the left branch of $\tau(0^{\ell} \cdot \sigma_{l-k}(e_{k+1},e_{k+2},\ldots, e_n))$ if $e_2=0$.  If $e_2=1$, then the number of black nodes in the left branch of $\tau(e)$ exceeds  the number of black nodes in the left branch of $\tau(0^{\ell}\cdot \sigma_{\ell-k}(e_{k+1},e_{k+2},\ldots, e_n))$ by one. The result follows inductively.
\end{proof}

The mapping $\tau$ visibly encodes the number of initial zeros in an inversion sequence.  The following corollary follows quickly from the definition of $\tau$.

\begin{corollary}
Let $e \in \I_n(021)$ where $0=e_1=e_2=\cdots=e_{\ell+1}$. 
 Then the top $\ell$ nodes of the leftmost branch of $\tau(e)$ are white.
\end{corollary}

\section{Concluding remarks}

This report only scratches the surface of pattern-avoidance in inversion sequences, but we hope that it demonstrates the surprising ability of pattern-avoiding inversion sequences to provide simple and natural models for familiar combinatorial sequences.

A number of open questions have been raised throughout this paper; most notably,  finding general expressions for the avoidance sequences for the patterns $120$, $110$, and $010$.
In addition to settling Conjecture 1 in Section 3.1 and Conjecture 2 in Section 4.1, it would be nice to see a simple bijection between $\I_n(012)$ and Boolean permutations of $[n]$ (Section 2.1); a combinatorial interpretation of the recurrence at the end of Section 3.4 for the number of inversion sequences in $\I_n(101)$ with $k$ zeros;  and a direct combinatorial proof of Corollary 4 in Section 4.2.

There are many obvious ways to extend or generalize this work:  consider longer patterns, sets of patterns, other statistics,  $q$-analogs, bijective proofs, or more natural bijections;  settle enumeration questions;  or even  consider pattern avoidance in $\I_n^{(s)}$ for other sequences $s$ of positive integers.

In Part II of this report, we consider a different generalization of pattern-avoidance in inversion sequences and discover some nice surprises, a few conjectures, and several open questions.

\vspace{.2in}
\noindent
{\large {\bf Acknowledgements}} 

We would like to express our appreciation to Ira Gessel for a generating function argument that led to the bijection between 021-avoiding inversion sequences and the Large Schr\"oder numbers in Section 4.1.
Thanks to Michael Albert for making his PermLab software freely available \cite{PermLab1.0}.
Thanks to the Simons Foundation for a grant to the third author which supported the travel of the second author for collaboration.  We extend our thanks to Toufik Mansour for sending us an advance copy of \cite{Mansour}.

We especially owe a debt of gratitude to Neil Sloane and the OEIS Foundation, Inc. Our work was greatly facilitated by the On-Line Encyclopedia of Integer Sequences \cite{Sloane}.

\bibliographystyle{plain}
\bibliography{pais}

\begin{thebibliography}{10}

\bibitem{PermLab1.0}
Michael Albert.
\newblock Permlab: Software for permutation patterns, 2012.
\newblock \\\url{http://www.cs.otago.ac.nz/staffpriv/malbert/permlab.php}.

\bibitem{Albert}
Michael Albert, Cheyne Homberger, and Jay Pantone.
\newblock Equipopularity classes in the separable permutations.
\newblock {\em Electron. J. Combin.}, 22(2):\#P2.2, 2015.

\bibitem{Bousquet}
Mireille Bousquet-M\'{e}lou, Anders Claesson, Mark Dukes, and Sergey Kitaev.
\newblock (2+2)-free posets, ascent sequences and pattern avoiding
  permutations.
\newblock {\em J. Combin. Theory, Ser. A}, 117(7):884 -- 909, 2010.

\bibitem{callan10}
David Callan.
\newblock A bijection to count (1-23-4)-avoiding permutations.
\newblock 2010.
\newblock Unpublished manuscript, {\url{http://arxiv.org/abs/1008.2375}}.

\bibitem{ChowShiu}
Chak-On Chow and Wai~Chee Shiu.
\newblock Counting simsun permutations by descents.
\newblock {\em Ann. Comb.}, 15(4):625--635, 2011.

\bibitem{drake}
Brian Drake.
\newblock {\em An inversion theorem for labeled trees and some limits of areas
  under lattice paths}.
\newblock ProQuest LLC, Ann Arbor, MI, 2008.
\newblock Thesis (Ph.D.)--Brandeis University.

\bibitem{DS}
Paul Duncan and Einar Steingr{\'{\i}}msson.
\newblock Pattern avoidance in ascent sequences.
\newblock {\em Electron. J. Combin.}, 18(1):Paper 226, 17, 2011.

\bibitem{elizalde}
Sergi Elizalde.
\newblock Asymptotic enumeration of permutations avoiding generalized patterns.
\newblock {\em Adv. in Appl. Math.}, 36(2):138--155, 2006.

\bibitem{gessel}
Ira Gessel.
\newblock Schr{\"o}der numbers, large and small.
\newblock {\em Cana{D}{A}{M} 2009, {M}ontreal}, 2009.
\newblock Slides, {\\
  \url{http://www.crm.umontreal.ca/CanaDAM2009/pdf/gessel.pdf}}.

\bibitem{Sloane}
OEIS~Foundation Inc.
\newblock The {O}n-{L}ine {E}ncyclopedia of {I}nteger {S}equences, 2011.
\newblock \url{http://oeis.org}.

\bibitem{kitaev}
Sergey Kitaev.
\newblock {\em Patterns in permutations and words}.
\newblock Monographs in Theoretical Computer Science. An EATCS Series.
  Springer, Heidelberg, 2011.
\newblock With a foreword by Jeffrey B. Remmel.

\bibitem{knuth1}
Donald~E. Knuth.
\newblock {\em The art of computer programming. {V}ol. 1: {F}undamental
  algorithms}.
\newblock Second printing. Addison-Wesley Publishing Co., Reading,
  Mass.-London-Don Mills, Ont, 1969.

\bibitem{knuth3}
Donald~E. Knuth.
\newblock {\em The art of computer programming. {V}olume 3}.
\newblock Addison-Wesley Publishing Co., Reading, Mass.-London-Don Mills, Ont.,
  1973.
\newblock Sorting and searching, Addison-Wesley Series in Computer Science and
  Information Processing.

\bibitem{KPP}
A.~G. Kuznetsov, I.~M. Pak, and A.~E. Postnikov.
\newblock Increasing trees and alternating permutations.
\newblock {\em Uspekhi Mat. Nauk}, 49(6(300)):79--110, 1994.

\bibitem{macmahon}
Percy~A. MacMahon.
\newblock {\em Combinatory analysis}.
\newblock Two volumes (bound as one). Chelsea Publishing Co., New York, 1960.

\bibitem{Mansour}
Toufik Mansour and Mark Shattuck.
\newblock Pattern avoidance in inversion sequences.
\newblock Preprint (private communication).

\bibitem{MS}
Toufik Mansour and Mark Shattuck.
\newblock Some enumerative results related to ascent sequences.
\newblock {\em Discrete Math.}, 315:29--41, 2014.

\bibitem{Ptenner}
T.~Kyle Petersen and Bridget~Eileen Tenner.
\newblock The depth of a permutation.
\newblock {\em J. Comb.}, 6(1-2):145--178, 2015.

\bibitem{Rotem}
D.~Rotem.
\newblock Stack sortable permutations.
\newblock {\em Discrete Math.}, 33(2):185--196, 1981.

\bibitem{SS}
Carla~D. Savage and Michael~J. Schuster.
\newblock Ehrhart series of lecture hall polytopes and {E}ulerian polynomials
  for inversion sequences.
\newblock {\em J. Combin. Theory Ser. A}, 119(4):850--870, 2012.

\bibitem{SV}
Carla~D. Savage and Mirk{\'o} Visontai.
\newblock The {$\bold{s}$}-{E}ulerian polynomials have only real roots.
\newblock {\em Trans. Amer. Math. Soc.}, 367(2):1441--1466, 2015.

\bibitem{shapiro}
Louis Shapiro and A.~B. Stephens.
\newblock Bootstrap percolation, the {S}chr\"oder numbers, and the {$N$}-kings
  problem.
\newblock {\em SIAM J. Discrete Math.}, 4(2):275--280, 1991.

\bibitem{SimionSchmidt}
Rodica Simion and Frank~W. Schmidt.
\newblock Restricted permutations.
\newblock {\em European J. Combin.}, 6(4):383--406, 1985.

\bibitem{sundaram}
Sheila Sundaram.
\newblock The homology of partitions with an even number of blocks.
\newblock {\em J. Algebraic Combin.}, 4(1):69--92, 1995.

\bibitem{tenner}
Bridget~Eileen Tenner.
\newblock Pattern avoidance and the {B}ruhat order.
\newblock {\em J. Combin. Theory Ser. A}, 114(5):888--905, 2007.

\bibitem{west}
Julian West.
\newblock Generating trees and the {C}atalan and {S}chr\"oder numbers.
\newblock {\em Discrete Math.}, 146(1-3):247--262, 1995.

\end{thebibliography}

\end{document}